\documentclass[11pt,letterpaper]{amsart}
\usepackage{import}
\usepackage{mathabx}
\usepackage{bbm}
\usepackage{fontenc}
\usepackage[T1]{fontenc}
\usepackage{amsfonts}
\usepackage{mathrsfs}
\usepackage{stmaryrd}
\usepackage{upgreek}
\usepackage{textcomp}

\usepackage{extarrow}
\usepackage{enumerate}
\usepackage{graphicx}
\usepackage[colorlinks=true, linkcolor=blue]{hyperref}
\usepackage{latexsym}
\usepackage{mathtools}
\usepackage{mathtext}
\usepackage{amsmath}
\usepackage{amssymb}
\usepackage{amsthm}
\usepackage{tikz}
\usepackage[all,cmtip]{xy}
\usepackage[mathscr]{eucal}
\usepackage[toc,page]{appendix}
\usepackage{a4wide}

\usetikzlibrary{arrows}
\usetikzlibrary{matrix}
\usetikzlibrary{shapes}
\usetikzlibrary{snakes}
\usetikzlibrary{matrix}

\DeclareMathOperator{\Spec}{\textup{Spec}}

\DeclareMathOperator{\et}{\textup{et}}

\usepackage{biblatex}
\theoremstyle{plain}
\newtheorem{thm}{Theorem}[section]
\newtheorem*{thm*}{Theorem}
\newtheorem{lem}[thm]{Lemma}

\theoremstyle{remark}
\newtheorem{rmk}[thm]{Remark}%[section] \setcounter{rmk}{0}

\newtheorem*{rmk*}{Remark}

\theoremstyle{definition}
\newtheorem*{const*}{Construction}

\theoremstyle{plain}
\newtheorem{thmI}{Theorem}
\newtheorem{thmII}{Theorem}
\def\cov{{\rm Cov}}

\def\eet{{\textup{Ét}}}

\begin{document}
\title[Comparison]{A Comparison Theorem For The Pro-étale Fundamental
Group}
\author{Jiu-Kang Yu, Lei Zhang}

\date{\today}

\global\long\def\A{\mathbb{A}}

\global\long\def\Ab{(\textup{Ab})}

\global\long\def\C{\mathbb{C}}

\global\long\def\Cat{(\textup{Cat})}

\global\long\def\Di#1{\textup{D}(#1)}

\global\long\def\E{\mathbb{E}}

\global\long\def\F{\mathbb{F}}

\global\long\def\GCov{G\textup{-Cov}}

\global\long\def\Gcat{(\textup{Galois cat})}

\global\long\def\Gfsets#1{#1\textup{-fsets}}

\global\long\def\Gm{\mathbb{G}_{m}}

\global\long\def\GrCov#1{\textup{D}(#1)\textup{-Cov}}

\global\long\def\Grp{(\textup{Grps})}

\global\long\def\Gsets#1{(#1\textup{-sets})}

\global\long\def\HCov{H\textup{-Cov}}

\global\long\def\MCov{\textup{D}(M)\textup{-Cov}}

\global\long\def\MHilb{M\textup{-Hilb}}

\global\long\def\N{\mathbb{N}}

\global\long\def\PGor{\textup{PGor}}

\global\long\def\PGrp{(\textup{Profinite Grp})}

\global\long\def\PP{\mathbb{P}}

\global\long\def\Pj{\mathbb{P}}

\global\long\def\Q{\mathbb{Q}}

\global\long\def\RCov#1{#1\textup{-Cov}}

\global\long\def\RR{\mathbb{R}}

\global\long\def\Sch{\textup{Sch}}

\global\long\def\WW{\textup{W}}

\global\long\def\Z{\mathbb{Z}}

\global\long\def\acts{\curvearrowright}

\global\long\def\alA{\mathscr{A}}

\global\long\def\alB{\mathscr{B}}

\global\long\def\arr{\longrightarrow}

\global\long\def\arrdi#1{\xlongrightarrow{#1}}

\global\long\def\catC{\mathscr{C}}

\global\long\def\catD{\mathscr{D}}

\global\long\def\catF{\mathscr{F}}

\global\long\def\catG{\mathscr{G}}

\global\long\def\comma{,\ }

\global\long\def\covU{\mathcal{U}}

\global\long\def\covV{\mathcal{V}}

\global\long\def\covW{\mathcal{W}}

\global\long\def\duale#1{{#1}^{\vee}}

\global\long\def\fasc#1{\widetilde{#1}}

\global\long\def\fsets{(\textup{f-sets})}

\global\long\def\iL{r\mathscr{L}}

\global\long\def\id{\textup{id}}

\global\long\def\la{\langle}

\global\long\def\odi#1{\mathcal{O}_{#1}}

\global\long\def\ra{\rangle}

\global\long\def\set{(\textup{Sets})}

\global\long\def\sets{(\textup{Sets})}

\global\long\def\shA{\mathcal{A}}

\global\long\def\shB{\mathcal{B}}

\global\long\def\shC{\mathcal{C}}

\global\long\def\shD{\mathcal{D}}

\global\long\def\shE{\mathcal{E}}

\global\long\def\shF{\mathcal{F}}

\global\long\def\shG{\mathcal{G}}

\global\long\def\shH{\mathcal{H}}

\global\long\def\shI{\mathcal{I}}

\global\long\def\shJ{\mathcal{J}}

\global\long\def\shK{\mathcal{K}}

\global\long\def\shL{\mathcal{L}}

\global\long\def\shM{\mathcal{M}}

\global\long\def\shN{\mathcal{N}}

\global\long\def\shO{\mathcal{O}}

\global\long\def\shP{\mathcal{P}}

\global\long\def\shQ{\mathcal{Q}}

\global\long\def\shR{\mathcal{R}}

\global\long\def\shS{\mathcal{S}}

\global\long\def\shT{\mathcal{T}}

\global\long\def\shU{\mathcal{U}}

\global\long\def\shV{\mathcal{V}}

\global\long\def\shW{\mathcal{W}}

\global\long\def\shX{\mathcal{X}}

\global\long\def\shY{\mathcal{Y}}

\global\long\def\shZ{\mathcal{Z}}

\global\long\def\st{\ | \ }

\global\long\def\stA{\mathcal{A}}

\global\long\def\stB{\mathcal{B}}

\global\long\def\stC{\mathcal{C}}

\global\long\def\stD{\mathcal{D}}

\global\long\def\stE{\mathcal{E}}

\global\long\def\stF{\mathcal{F}}

\global\long\def\stG{\mathcal{G}}

\global\long\def\stH{\mathcal{H}}

\global\long\def\stI{\mathcal{I}}

\global\long\def\stJ{\mathcal{J}}

\global\long\def\stK{\mathcal{K}}

\global\long\def\stL{\mathcal{L}}

\global\long\def\stM{\mathcal{M}}

\global\long\def\stN{\mathcal{N}}

\global\long\def\stO{\mathcal{O}}

\global\long\def\stP{\mathcal{P}}

\global\long\def\stQ{\mathcal{Q}}

\global\long\def\stR{\mathcal{R}}

\global\long\def\stS{\mathcal{S}}

\global\long\def\stT{\mathcal{T}}

\global\long\def\stU{\mathcal{U}}

\global\long\def\stV{\mathcal{V}}

\global\long\def\stW{\mathcal{W}}

\global\long\def\stX{\mathcal{X}}

\global\long\def\stY{\mathcal{Y}}

\global\long\def\stZ{\mathcal{Z}}

\global\long\def\then{\ \Longrightarrow\ }

\global\long\def\L{\textup{L}}

\global\long\def\l{\textup{l}}

%-------------------------------------------
% Sonderbuchstaben mit Doppellinie
\newcommand{\B}{{\mathbb B}}
\newcommand{\D}{{\mathbb D}}
\newcommand{\G}{{\mathbb G}}
\renewcommand{\H}{{\mathbb H}}
\newcommand{\I}{{\mathbb I}}
\newcommand{\J}{{\mathbb J}}
\newcommand{\M}{{\mathbb M}}
\renewcommand{\P}{{\mathbb P}}
\newcommand{\R}{{\mathbb R}}
\newcommand{\T}{{\mathbb T}}
\newcommand{\U}{{\mathbb U}}
\newcommand{\V}{{\mathbb V}}
\newcommand{\W}{{\mathbb W}}
\newcommand{\X}{{\mathbb X}}
\newcommand{\Y}{{\mathbb Y}}

%Skriptbuchstaben
\newcommand{\sA}{{\mathcal A}}
\newcommand{\sB}{{\mathcal B}}
\newcommand{\sC}{{\mathcal C}}
\newcommand{\sD}{{\mathcal D}}
\newcommand{\sE}{{\mathcal E}}
\newcommand{\sF}{{\mathcal F}}
\newcommand{\sG}{{\mathcal G}}
\newcommand{\sH}{{\mathcal H}}
\newcommand{\sI}{{\mathcal I}}
\newcommand{\sJ}{{\mathcal J}}
\newcommand{\sK}{{\mathcal K}}
\newcommand{\sL}{{\mathcal L}}
\newcommand{\sM}{{\mathcal M}}
\newcommand{\sN}{{\mathcal N}}
\newcommand{\sO}{{\mathcal O}}
\newcommand{\sP}{{\mathcal P}}
\newcommand{\sQ}{{\mathcal Q}}
\newcommand{\sR}{{\mathcal R}}
\newcommand{\sS}{{\mathcal S}}
\newcommand{\sT}{{\mathcal T}}
\newcommand{\sU}{{\mathcal U}}
\newcommand{\sV}{{\mathcal V}}
\newcommand{\sW}{{\mathcal W}}
\newcommand{\sX}{{\mathcal X}}
\newcommand{\sY}{{\mathcal Y}}
\newcommand{\sZ}{{\mathcal Z}}

%-----------------------------------------------

\newcommand{\Aff}{{\rm Aff}}
\newcommand{\Aut}{{\rm Aut}}
\newcommand{\an}{{\rm an}}
\newcommand{\Bd}{{\rm Band}}
\newcommand{\Cats}{{\rm Cats}}
\newcommand{\ch}{\textup{Ch}}
\newcommand{\Char}{{\rm char}}
\newcommand{\codim}{{\rm codim}}
\newcommand{\cont}{{\rm cont}}
\newcommand{\Cov}{\textup{Cov}}
\newcommand{\Crys}{{\rm Crys}}
\newcommand{\cts}{\textup{cts}}
\newcommand{\Div}{{\rm Div}}
\newcommand{\Dmod}{{\rm Dmod}}
\newcommand{\ECov}{{\rm ECov}}
\newcommand{\ed}{{\rm ed}}
\newcommand{\Ess}{{\rm EFin}}
\renewcommand{\et}{\textup{\'et}}
\newcommand{\ev}{\textup{ev}}
\newcommand{\Fdiv}{{\rm Fdiv}}
\newcommand{\Fib}{{\rm Fib}}
\newcommand{\FSets}{{\rm FSets}}
\newcommand{\FtAff}{{\rm FtAff}}
\newcommand{\Gal}{{\rm Gal}}
\newcommand{\height}{\textup{ht}}
\newcommand{\Hom}{{\rm Hom}}
\newcommand{\iinf}{\textup{inf}}
\newcommand{\im}{{\rm im}}
\newcommand{\Ker}{{\rm Ker}}
\newcommand{\LL}{\textup{L}}
\newcommand{\Loc}{{\rm Loc}}
\newcommand{\Max}{{\rm Max \ }}
\newcommand{\MIC}{\mbox{MIC}}
\newcommand{\Min}{{\rm Min \ }}
\newcommand{\NN}{\textup{N}}
\newcommand{\Mod}{\text{\sf Mod}}
\newcommand{\Noohi}{\textup{Noohi}}
\newcommand{\perf}{\textup{perf}}
\newcommand{\pet}{{\textup{proét}}}
\newcommand{\Pic}{{\rm Pic}}
\newcommand{\Rep}{\text{\sf Rep}}
\newcommand{\Res}{{\rm Res}}
\newcommand{\rank}{{\rm rank}}
\newcommand{\red}{{\rm red}}
\newcommand{\Sets}{\textup{Sets}}
\newcommand{\Spf}{\textup{Spf}}
\newcommand{\spe}{\textup{sp}}
\newcommand{\str}{\textup{str}}
\newcommand{\strat}{{\rm Str}}
\newcommand{\sym}{\text{Sym}}
\newcommand{\tp}{{\rm top}}
\newcommand{\Tr}{{\rm Tr}}
\newcommand{\trace}{{\rm Tr}}
\newcommand{\vect}{\text{\sf vect}}
\newcommand{\Vect}{\text{\sf Vect}}

% If the folder containing packages_and_functions.for_preamble.tex is the
% current folder, then use the following, and comment the subimport.

%\input{packages_and_functions}

% The four lines are together. This is todo setup step 2.
 %\makeatletter 
 %\providecommand\@dotsep{5} 
 %\makeatother 
 %\listoftodos\relax
 
\address{ 
Jiu-Kang YU\\
     The Chinese University of Hong Kong\\
     Institute of Mathematical Sciencce (IMS)\\    
    Shatin, New Territories\\ Hong Kong }
\email{jkyu@ims.cuhk.edu.hk}

 \address{Lei ZHANG\\
     Sun Yat-Sen University\\ School of Mathematics
     (Zhuhai)\\Zhuhai, Guangdong, P.~R.~China}
\email{cumt559@gmail.com}

\setcounter{section}{0}
\maketitle

\begin{abstract} Let
    $X$ be a connected scheme locally of finite type over $\C$, and let $X^\an$
    be its associated analytic space. In
    this paper, we provide a comparison map from the topological
fundamental group of $X^\an$
to the pro-étale fundamental group of $X$.\end{abstract}

\section*{Introduction}
Let $X$ be a connected, locally path-connected and semi-locally simply connected topological space, and let $x_0 \in X$ be a point in $X$. Then
$\pi^{\tp}_1(X,x_0)$ classifies all the covering spaces of $X$, that is,
there is an equivalence
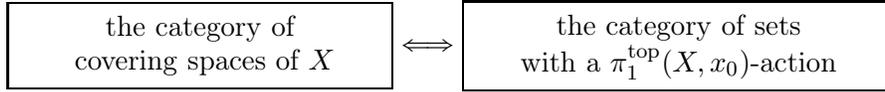
\[ 
    \fbox{
  \begin{tabular}{ccc}
            &\text{the category of}&\\
      &\text{ covering spaces of $X$}&
  \end{tabular}
}
\Longleftrightarrow
\fbox{
  \begin{tabular}{ccc}
            &\text{the category of sets}&\\
            &\text{with a $\pi_1^\tp(X,x_0)$-action}&
  \end{tabular}
}
\]

Inspired by the analogy between the above classification and the classical
Galois theory, A.~Grothendieck constructed in SGA1 \textit{the étale fundamental group} $\pi_1^\et(X,x_0)$  for each connected scheme $X$ and a
geometric point $x_0\in X$. The fundamental group $\pi_1^\et(X,x_0)$
classifies all the finite étale covers of $X$, \emph{i.e.} there is an
equivalence 
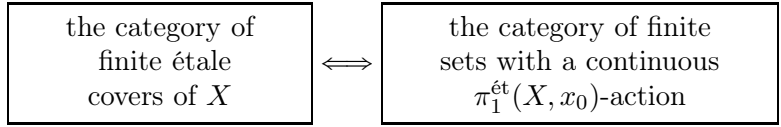
\[ 
    \fbox{
  \begin{tabular}{ccc}
            &\text{the category of}&\\
            &\text{ finite étale}&\\
            &\text{covers of $X$}&
  \end{tabular}
}
\Longleftrightarrow
\fbox{
  \begin{tabular}{ccc}
            &\text{the category of finite}&\\
            &\text{sets with a continuous}&\\
            &\text{$\pi_1^\et(X,x_0)$-action}&
  \end{tabular}
}
\]
A.~Grothendieck also showed that if $X$ is a scheme locally of finite type over
$\C$, then
$\pi_1^\et(X,x_0)$ is the profinite completion of $\pi_1^\tp(X^\an,x_0)$,
where $X^\an$ is the complex analytic space associated with $X$. There are
two main ingredients behind this comparison:
\begin{itemize}
\item If $p\colon Y\to X$ is a morphism of $\C$-schemes
    locally of finite type, then the analytification $p^\an$ is a finite covering
    space if and only if $p$ is a finite étale cover
    (cf.~\cite[Exposé XII, Proposition 3.1 (iii), Proposition 3.2 (vi)]{SGA1});
\item The \textit{Riemann Existence Theorem}
(\cite[Exposé XII, Théorème 5.1, p.~251]{SGA1}) which states that 
the analytification functor $(-)^\an$ from the category of schemes locally of
finite type over $X$ to the category of analytic spaces over $X^\an$ induces an equivalence between the finite
étale covers of $X$ and the finite covering spaces of $X^\an$.
\end{itemize}

In \cite{BS15}, B.~Bhatt and P.~Scholze introduced the notion of \emph{the pro-étale
fundamental group} $\pi^\pet(X,x_0)$, a
topological group
which classifies the {\it geometric covers} of \(X\) (cf.~\cite[Definition
7.3.1]{BS15}).  The geometric covers are locally constant sheaves in the
pro-étale topology, in particular, they include the finite \'etale
covers.  Therefore, \(\pi_1^\pet(X,x_0)\) refines Grothendieck's \'etale
fundamental group \(\pi_1^\et(X,x_0)\), which classifies the finite
\'etale covers.  In fact, there is a natural morphism
\(\pi_1^\pet(X,x_0)\to \pi_1^\et(X,x_0)\) which makes
\(\pi_1^\et(X,x_0)\) the
profinite completion of \(\pi_1^\pet(X,x_0)\). However, there has been no comparison
between $\pi^\tp$ and $\pi^\pet$.
The main purpose of this paper to achieve this comparison. 

Let $X$ be a scheme locally of finite type over $\C$. Recall that if $p'\colon
Y'\to X^\an$ is a map of analytic spaces, then we say $p'$ is
\emph{algebraizable} if there is a map $p\colon Y\to X$ of $\C$-schemes locally of
finite type whose analytification
(cf.~\cite[Théorème et définition 1.1]{SGA1}) $p^\an$ is $p'$.

\begin{thmI} {\rm(cf.~\ref{geom.covers are covering spaces}, \ref{covering spaces
    are geometric covers} and \ref{full faithfulness})} Let $X$ be a scheme locally of finite type over $\C$.
    \begin{enumerate}
        \item If $p\colon Y\to X$ is a morphism of $\C$-schemes
    locally of finite type, then its analytification $p^\an$ is a covering
    space if and only if $p$ is a geometric cover;
\item The analytification functor $(-)^\an$ from the category of schemes locally of
finite type over $X$ to the category of analytic spaces over $X^\an$ induces an
equivalence between the geometric covers of $X$ and the algebraizable covering spaces of $X^\an$.
\end{enumerate}
Moreover, the full subcategory of algebraizable covering spaces is stable
under subobjects and quotients.
\end{thmI}
Therefore, we obtain a natural comparison map.
\begin{thmII}\label{thmII}{\rm(cf.~\ref{comparison map})} Let $X$ be a connected scheme
    locally of finite type over $\C$, and let $x_0\in X$ be a geometric
    point. Then there is a natural comparison
    map
    \[\pi_1^\tp(X^\an,x_0)\longrightarrow\pi_1^\pet(X,x_0)\] 
whose image is dense. Moreover, there is no nontrivial continuous
homomorphism from $\pi_1^\pet(X,x_0)$ to a discrete group $G$ whose
restriction to $\pi_1^\tp(X^\an,x_0)$ is trivial.
\end{thmII}
Thus, the difference between $\pi^\tp$
and $\pi^\pet$ for complex varieties lies exactly at those covering spaces,
such as the exponential map
$\G_{a}^\an\to \G_{m}^\an$, which are not algebraizable.

It is remarkable that unlike the étale fundamental group, which is the profinite
completion of the topological fundamental group, the pro-étale
fundamental group is not determined by the topological fundamental group. For
example, let $X$ be the space $\P^1_\C$ with 0 and $\infty$ identified,
then both $\pi_1^\tp(X^\an)$ and $\pi^\tp_1(\G_{m,\C}^\an)$ are $\Z$. However,
$\pi_1^\pet(X)=\Z$ (cf.~\cite[Theorem IV]{LYZ21}) while $\pi_1^\pet(\G_{m,\C})=\hat{\Z}$. Indeed, the
comparison map $\pi^\tp_1(X^\an,x_0)\to \pi^\pet_1(X,x_0)$ is an
isomorphism by Theorem \ref{thmII}, so all the covering spaces of $X^\an$ are
algebraizable.

\subsection*{Notation and conventions.} For analytic spaces (always over \(\mathbb{C}\)), we
follow the convention of \cite[Définition 2.1]{Gro60} and \cite[Expos\'e XII]{SGA1}, allowing
non-separated analytic spaces.  For topological spaces, we follow the convention
of
\cite[\href{https://stacks.math.columbia.edu/tag/004C}{004C}]{stacks-project}.  In particular, a proper map is
assumed to be separated \cite[\href{https://stacks.math.columbia.edu/tag/005M}{005M}]{stacks-project}.

Let \(p\colon Y\to X\) be a map of topological spaces (resp.~analytic
spaces).  The map $p$ is called \emph{étale} if it is a
local isomorphism.   We denote by $\eet(X)$ the category of étale maps
with target $X$.   We denote by \(\cov(X)\) the category of covering spaces
over \(X\) (resp.~the full subcategory of $\eet(X)$ consisting of maps $p\colon
Y\to X$ which are covering maps for the underlying topological spaces). If $X$ is a
scheme, then $\Cov(X)$ denotes the category of geometric covers of
$X$.

A covering map \(p:Y\to X\) is called a {\it trivial covering space}
if it is isomorphic (in \(\cov(X)\)) to \(\coprod_J X \to X\) for some
set \(J\), where \(\coprod_J X\) denote the disjoint union of copies
of \(X\) indexed by \(J\).  Some authors would say that \(X\) is
evenly covered by \(p\).

\section{Geometric Covers are Covering Spaces}

\begin{lem}\label{etale analytique space and etale top space}
    \begin{enumerate}
    \item If $X$ is an analytic space, then the
    forgetful functor from the category of étale analytic spaces over $X$ to the category
    of topological spaces over $X$ is fully faithful;
    \item If $u\colon X'\to X$ is a map of analytic spaces such that the
        induced map on the underlying topological spaces is a
        homoemorphism, then the
        pullback along $u$ induces an equivalence
        $\eet(X)\xrightarrow{\simeq}\eet(X')$.
    \end{enumerate}
\end{lem}

\begin{proof} Let $f_1\colon (Y_1,\sO_{Y_1})\to (X,\sO_X)$ and $f_2\colon
    (Y_2,\sO_{Y_2})\to (X,\sO_X)$ be two étale maps of analytic
    spaces. We want to show that if $g$ is a morphism of topological
    spaces $Y_1\to Y_2$ over $X$, then there
    exists a unique map of sheaves $g^{-1}\sO_{Y_2}\to \sO_{Y_1}$   making the
    following diagram 
    \[
        \xymatrix{
            g^{-1}\sO_{Y_2}\ar@{..>}[rr]&&\sO_{Y_1}\\
                                        &f_1^{-1}\sO_{X}\ar[ul]^{g^{-1}(f_2^{\#})}\ar[ur]_{f_1^{\#}}
        }
    \]
    commutative \emph{i.e.} making $g$
    a map of $X$-analytic spaces $f_1\to f_2$. The uniqueness is local on
    $Y_1$.
    Since $f_1,f_2$ are local isomorphisms, to prove the uniqueness we may assume that they are
    isomorphisms. Then the situation is clear. Having the uniqueness,
    the existence becomes local on $Y_1$, thus assuming that  $f_1, f_2$ are
    isomorphisms again we finish the proof of (1).

    To prove part (2), one first notices that (1) already provides the full
    faithfulness, so we only have to show the essential surjectivity.
    Suppose that $f'\colon (Y',\sO_{Y'})\to (X',\sO_{X'})$ is an étale map.
    We want to find an étale map $f\colon (Y,\sO_Y)\to(X,\sO_X)$ whose
    pullback along $u$ is $f'$. It follows from the full faithfulness
    of the pullback that the existence of $f$ can be proved locally on $Y'$.
    Thus we may assume that $f'$ is an isomorphism, in which case the
    assertion is clear.
\end{proof}
    
\begin{rmk}\label{analytic space vs topological space} \begin{enumerate}\item An étale map of analytic spaces induces an étale map of the
            underlying topological spaces, but not vise versa.
    \item It follows from \ref{etale analytique space and etale top space} (1)
that if $X$ is an analytic spaces, then any map in $\Cov(X)$ is locally on
$X$ a trivial covering space -- as an analytic space!  \end{enumerate}
\end{rmk}

\begin{lem}\label{what proper maps are good for} Let $f\colon \tilde{X}\to X$ be a closed map of topological spaces, and let
    $x\in X$ be a point. If
    $V\subseteq \tilde{X}$ is an open subset containing $f^{-1}(x)$, then there exists
    an open neighborhood $U\subseteq X$ of $x$ such that $f^{-1}(U)\subseteq
    V$. 
\end{lem}
\begin{proof}
    Set $U\coloneq X\setminus{f(\tilde{X}\setminus{V})}$.
\end{proof}

%\begin{lem}\leavevmode
%    \begin{enumerate}
%    \item Finite étale maps are covering spaces (for topological/analytic
%        spaces); 
%    \item The category $Cov(X)$ has arbitrary disjoint union if
%        $X$ is locally contractible.
% \end{enumerate}
%\end{lem}
% \begin{proof}
% \end{proof}

 \begin{lem}\label{locally trivial}
     Let $X$ be a scheme locally of finite type over $\C$,
and let $p\colon Y\to X$ be a geometric cover which is a disjoint union of finite étale covers. Then for each $x\in X^\an$ there is an open
neighborhood $U\subseteq X^\an$ such that $p^\an|_{U}$ is a
trivial covering space.
\end{lem}
\begin{proof}
    An analytic space is locally simply connected because it is locally a
    CW-complex (cf.~\cite[Theorem 1]{Loj64}) and every CW-complex is locally
    contractible (cf.~\cite[A.4, p. 522]{Hat02}). Thus we can choose $U$ to be a simply
    connected open neighborhood of $x$. Then the claim follows from
    Remark \ref{analytic space vs topological space} (2).
\end{proof}

% Let $p\colon Y\to X$ be a map of topological spaces. Recall that $p$ is called étale if
% it is a local homoemorphism. We denote $\eet(X)$ the category of étale maps
% with target $X$. The map $p$ is called proper if it is universally closed.
 
\begin{lem}\label{closed van kampen}
Let $X$ be a topological space,  and let
  \(Z_1, Z_2\) be closed subspaces of \(X\) such that
  \(X=Z_1\cup Z_2\). Then we have an equivalence given by the pullback functor:
  \[\eet(X)\xrightarrow{\ \ \cong\ \ }\eet(Z_1)\times_{\eet(Z)}\eet(Z_2)\]
  where $Z=  Z_1\cap Z_2$.
\end{lem}
\begin{proof}  We remark that since a proper map is separated,
  \(\Delta _f\) is indeed a closed embedding.
  Set $Y\coloneq Z_1\coprod Z_2$.
    By \cite[{}4.1,4.7]{JT94} or \cite[Theorem 5.6]{Ver94} a
    Bourbaki-proper surjective map (simply called a proper map in
    loc.~cit.) such as
    $\phi\colon  Y\arr X$ is an effective étale descent
    map. Thus
    an object in $\eet(X)$ is equivalent to an object in
    $W\in\eet(Y)=\eet(Z_1)\times\eet(Z_2)$ together with an
    isomorphism $\phi\colon p_1^*W\xrightarrow{\cong}p_2^*W$ in $\eet(Y\times_XY)$
    satisfying cocycle conditions, where $p_i$ ($i=1,2$) are the two
    projections of $Y\times_XY$. Since \[Y\times_XY=Z_1\coprod Z_2\coprod
    (Z_1\cap Z_2)\coprod (Z_2\cap Z_1)\] the cocycle condition is equivalent to saying that
    $\phi|_{Z_1}=\id_{W|_{Z_1}}$, $\phi|_{Z_2}=\id_{W|_{Z_2}}$ and 
    \[
        \phi|_{Z_1\cap Z_2}\circ\phi|_{Z_2\cap Z_1}=\id_{W|_{Z}}.
    \]
    In other words, the pair $(W,\phi)$ is nothing but an object in
    $\eet(Z_1)\times_{\eet(Z)}\eet(Z_2)$.
\end{proof}

\begin{lem} \label{closed van kampen v2}
Let $X$ be a topological space. Let $Z\subseteq X$ be a closed subspace,
and suppose we have a proper surjective map
\[\tilde{X}\xrightarrow{\hspace{5pt}f\hspace{5pt}}X\]
Denote $f^{-1}(Z)$ by \(\tilde Z\). Suppose that the union
of  the images
of the two closed subspaces
\[
  \Delta_f:\tilde X\to \tilde X\times_X\tilde X,\quad
  \tilde Z\times_Z \tilde Z \to \tilde X \times_X  \tilde X
\]
is \(\tilde X\times _X \tilde X\).
Then the pullback functor induces an equivalence of categories:
\[\eet(X)\xrightarrow{\ \ \cong\ \ }\eet(\tilde{X})\times_{\eet(\tilde
Z)}\eet(Z).\]
\end{lem}

\begin{proof}
Let us construct a quasi-inverse of the pullback functor.  So we start with
a triple \((Y,W,\phi)\), where $Y\in\eet(\tilde{X})$, $W\in\eet(Z)$, and
$\phi: Y \times _{\tilde X}\tilde Z\xrightarrow{\cong} W\times_Z
\tilde Z$ is an isomorphism.
The idea is to construct a descent datum of \(Y\) for the morphism
\(f\), then we get the desired object in \(\eet(X)\) by applying
\cite[{}4.1, 4.7]{JT94} or \cite[Theorem 5.6]{Ver94}.

To construct this descent datum, we need to work on \(\tilde
X\times_X \tilde X\) using the full faithfulness part of the equivalence
\begin{equation}\label{two product}
\eet(\tilde{X}\times_X\tilde{X})\xrightarrow{\ \ \cong\ \ }
\eet(\tilde X)\times_{\eet(\tilde Z)}\eet\left(
\tilde Z\times_Z\tilde Z \right)
\end{equation}
which follows from Lemma \ref{closed van kampen}, because
the fiber product of the two closed immesions in the current lemma is the
closed immersion \(\tilde Z \xrightarrow{\Delta_f}\tilde
X\times_X\tilde X\).  Under \eqref{two product}, \(p_i^*(Y)\) corresponds
to the triple \((Y,q_i^*(Y|_{\tilde Z}),{\rm can}_i)\) for \(i=1,2\).
Here \(p_i:\tilde X\times_X \tilde X\to \tilde X\), \(q_i:\tilde Z \times_Z\tilde
Z\to Z\) are projections, and \({\rm can}_i\) is the obvious canonical
isomorphism.  The desired descent datum \(\lambda:p_1^*(Y) \to p_2^*(Y)\)
corresponds to \((Y,q_1^*(Y|_{\tilde Z}),{\rm
  can}_1)\to(Y,q_2^*(Y|_{\tilde Z}),{\rm can}_2)\) given by \({\rm
  id}_Y: Y \to Y\), \(\varphi: q_1^*(Y|_{\tilde Z})\to
q_2^*(Y|_{\tilde Z})\), where \(\varphi\) is the canonical descent
datum signifying the fact that \(Y|_{\tilde Z}\) is isomorphic to \(W
\times_Z \tilde Z\) via \(\phi\).

We have to show that \(\lambda\) is indeed a descent datum.  That is,
considering the triple fibred product, the fibred product and the projections:
\[
\begin{tikzpicture}[xscale=5,yscale=2.2]
    \node (A0_0) at (0, 0) {$\tilde{X}\times_X\tilde{X}\times_X\tilde{X}$};
    \node (A0_1) at (1, 0) {$\tilde{X}\times_X\tilde{X}$};
    \node (A0_2) at (2, 0) {$\tilde{X}$};
    
    \draw[>=latex,->] ([yshift= 2pt] A0_0.east) to
        [out=40,in=140]   node[above,scale=0.5]{$p_{12}$}   ([yshift= 2pt] A0_1.west);
    \draw[>=latex,->] (A0_0.east) -- node[above,midway,scale=0.5]{$p_{23}$} 
        (A0_1.west);
    \draw[>=latex,->] ([yshift=-2pt] A0_0.east) to
        [out=-40,in=-140]    node[below,scale=0.5]{$p_{13}$}   ([yshift=-2pt] A0_1.west);

    \draw[>=latex,->]                ([yshift= 1pt] A0_1.east)
        to [out=30,in=150] node[above,scale=0.5]{$p_{1}$}([yshift= 1pt] A0_2.west);
    \draw[>=latex,->] ([yshift=-1pt] A0_1.east) to
       [out=-30,in=-150] node[below,scale=0.5]{$p_{2}$} ([yshift=-1pt] A0_2.west);
\end{tikzpicture},
\]
we have to show the cocycle condition $p_{23}^*\lambda\circ
p_{12}^*\lambda=p_{13}^*\lambda$. Since the fibred
product is covered by $\Delta_f(\tilde{X})$ and
$\tilde Z \times_Z \tilde Z$, the triple fibred product is
covered by the triple diagonal $\Delta_f^3(\tilde{X})$ and the
closed subset
\(\tilde Z\times_Z\tilde Z \times_Z \tilde Z
    \).
Applying 
\ref{closed van kampen} we get an equivalence \begin{equation}\label{triple
    product}
\eet(\tilde{X}\times_X\tilde{X}\times_X\tilde{X})\xrightarrow{\ \ \cong\ \ }
\eet(\tilde X)\times_{\eet(\tilde Z)}\eet\left(
\tilde Z\times_Z\tilde Z\times_Z\tilde Z \right)
\end{equation} By the faithfulness part of \eqref{triple product}, it is enough to
check the equality
$p_{23}^*\lambda\circ
p_{12}^*\lambda=p_{13}^*\lambda$ on $\Delta_f^3(\tilde{X})$ and
$\tilde Z \times _Z \tilde Z \times_Z \tilde Z$ separately. On
$\Delta_f^3(\tilde{X})$ all the pullbacks of $\lambda$ are
identities, so there is nothing to check. On 
$\tilde Z \times _Z \tilde Z \times_Z \tilde Z$ the equality holds
because $\lambda|_{\tilde Z\times_Z \tilde Z}=\varphi$  is a descent
datum for \(\tilde Z \to Z\). 
\end{proof}

% We follow \cite{SGA1} for the definition of an analytic space. This is
% slightly different from the one in \cite{GAGA} where all the analytic
% spaces are assumed to be Hausdorff. 

\begin{thm}\label{geom.covers are covering spaces}

    Let $X$ be a scheme locally of finite type over $\C$, and let $p\colon Y\to X$ be a
    geometric cover. Then $p^\an$ is a covering space.

\end{thm}
\begin{proof}
Let $x\in X(\C)$ be a point. We want to find an open neighborhood of $x$ in
    $X^\an$
    over which $p^\an$ is a trivial covering space. For this we may
    and do assume
    that $X=\Spec(A)$, where $A$ is reduced (cf.~\ref{etale analytique space and etale
      top space} (2)) and of finite type.  Thus \(X\) is finite-dimensional.

We first remark that the theorem is true when \(X\) is normal.  Indeed
in this case 
\cite[{}7.4.10]{BS15} implies that every geometric cover \(Y\) is a
disjoint union of finite \'etale cover.  So the theorem follows from
\ref{locally trivial}.

We will apply induction on the dimension of $X$.  The theorem is
    trivial if \(\dim X=0\).  By the induction
    hypothesis, if $Z\subseteq X$ is the singular locus of $X$, then
    $p^\an|_{Z^\an}$ is
    a covering space.

    Let $f\colon \tilde{X}\to X$ be the normalization of \(X\)
    (cf.~\cite[\href{https://stacks.math.columbia.edu/tag/035E}{035E}]{stacks-project}). Then $f$ is finite, and it is an isomorphism away
    from $Z$. If $x\notin Z$, then the desired result follows from the
    first remark
    and \ref{locally trivial},
    so we may assume that $x\in Z$.

Write \(f^{-1}(x)=\{x_1,\ldots,x_n\}\subset\tilde{X}(\mathbb{C})\).
Since the topology of \({\tilde{X}}^\an\) is separated (\(\tilde{X}\) being affine), we
can find open neighborhoods \(V_i\) of \(x_i\) in \(\tilde X^\an\)
such that \(V_i\cap V_j=\emptyset\) for \(i\neq j\).  By \ref{what
  proper maps are good for}, we can find an open neighbordhood \(U\)
of \(x\) such that \({f^\an}^{-1}(U) \subset \coprod_{1\leq i\leq n}
V_i\).

We now choose \(\{V_i\}\) and \(U\) satisfying the conditions in the
preceding paragraph and in addition satifying that \(p^\an\) is trivial on each \(V_i\) and
on \(Z^\an \cap U\).  This is possible by our first remark, the
induction hypothesis, and
\ref{locally trivial}.

Write
\[
  p^\an|_{V_i}=\coprod_{J_i} V_i, \qquad p^\an|_{Z^\an \cap
    U}=\coprod_{J}  (Z^\an \cap U)
\]
for some sets \(J_i\) and \(J\).
Let \(W_i=V_i\cap f^{-1}(Z^\an \cap U)\) and let
\(\phi_i:(p^\an|_{V_i})|_{W_i} \to (p^\an|_{Z^\an \cap U})|_{W_i}\) be the
natural isomorphism.  Choose a {\it connected} open neighborhood \(A_i\) of
\(x_i\) in \(W_i\).  Then \(\phi_i|_{A_i}\) determines a bijection \(J_i
\simeq J\) (which may depend on the choice of \(A_i\)).  Changing the
sets \(J_i\) if necessary, we may and do
assume that \(J_i=J\) for each \(i\) and the bijection \(J_i=J\to J\)
induced by \(\phi|_{A_i}\) is the identity.

Finally, we write \(A_i=V_i'\cap f^{-1}(Z^\an \cap U)\) with \(V_i'\)
an open neighborhood of \(x_i\) in \(V_i\), and we find an open
neighborhood \(U'\) of \(x\) in \(U\) such that \({f^\an}^{-1}(U')
\subset \coprod_{1\leq i\leq n} V_i'\).  We now claim that
\(p^\an|_{U'}\) is a trivial covering space.

Write \(\tilde U'={f^\an}^{-1}(U')\) and \(Z_{U'}=Z^\an\cap U'\).
Applying \ref{closed van kampen v2} we have the following equivalence 
    \begin{equation}\label{trivial equivalence}
        \eet(U')\xrightarrow{\ \ \cong\ \ }
        \eet(\tilde U')\times_{\eet(\tilde{Z}_{U'})}\eet(Z_{U'})
    \end{equation}
    where $\tilde{Z}_{U'}\coloneq f^{-1}(Z_{U'})$.   We wish to show that
    \(p^\an|_{U'}\), an object in \(\eet(U')\), is a trivial cover.
    Consider the corresponding object on the right hand side of
    (\ref{trivial equivalence}), which is the triple
    \((p^\an|_{U'},p^\an|_{Z_{U'}},\phi)\).  We have
    \(p^\an|_{U'}=\coprod_J U'\), \(p^\an|_{Z_{U'}}=\coprod_J
    Z_{U'}\), and \(\phi\) is the identity since
    \(\phi\) is simply \(\phi_i\) on \({{f^\an}^{-1}(U')\cap
      V_i'}\subset A_i\).  Clearly this triple corresponds to a
    trivial cover in \(\eet(U')\).  This proves the claim and hence
    the theorem. \end{proof}

\section{Algebraizable Covering Spaces are Geometric Covers}
\begin{thm}\label{covering spaces are geometric covers} Let 
    $p\colon Y\to X$ be a map of schemes locally of finite type over $\C$. If $p^\an$ is a
    covering space of $X^\an$, then $p$ is a geometric cover of $X$. 
\end{thm}

\begin{proof} It is easy to see that a covering map is separated.  By \cite[Expoosé XII, Proposition 3.1 (iii), (viii)]{SGA1}, $p$
    is étale and separated. Since geometric covers are locally constant sheaves in the
    pro-étale topology, they satisfy effective pro-étale descent, so in
    particular effective Zariski descent. Thus we may assume that $X=\Spec(A)$, where $A$ of finite type.

    Now suppose that $R$ is a valuation ring with function field $K$, and
    suppose that we have the following commutative diagram
    \begin{equation}\label{homotopy lifting property}
        \begin{tikzpicture}[xscale=2.9,yscale=-1.2,
            baseline=(current bounding box.center)]
            \node (A0_0) at (0, 0) {$\Spec(K)$};
            \node (A0_1) at (1, 0) {$Y$};
            \node (A1_0) at (0, 1) {$\Spec(R)$};
            \node (A1_1) at (1, 1) {$X$};

            \draw[>=latex,->]                (A0_0) --
                node[auto]{$\scriptstyle{\eta}$}
                (A0_1);
          
            \draw[>=latex,->]                (A1_0) --
                node[auto]{$\scriptstyle{(x,\xi)}$}
                (A1_1);
            \draw[dashed,>=latex,->]                (A1_0) --
                node[auto]{}
                (A0_1);
            
            \path (A0_1) edge [->] node[right,scale=1]
                {$\scriptstyle{p}$} (A1_1);
            \path (A0_0) edge [->] node[auto] {$\scriptstyle{}$} (A1_0);
        \end{tikzpicture}
    \end{equation}
We have to show that there exists a unique dashed arrow as above which makes all the triangles
commutative.

If the morphism \(\Spec R\to X\) factors as \(\Spec R \to X'\to X\),
we can put \(Y'=Y\times_X X'\) and form the diagram
 \begin{equation}\label{homotopy lifting property2}
        \begin{tikzpicture}[xscale=2.9,yscale=-1.2,
            baseline=(current bounding box.center)]
            \node (A0_0) at (0, 0) {$\Spec(K)$};
            \node (A0_1) at (1, 0) {$Y'$};
            \node (A0_2) at (2,0) {$Y$};
            \node (A1_0) at (0, 1) {$\Spec(R)$};
            \node (A1_1) at (1, 1) {$X'$};
            \node (A1_2) at (2,1) {$X$};

            \draw[>=latex,->]                (A0_0) --
                node[auto]{}
                (A0_1);
         \draw[>=latex,->]                (A0_1) --
                node[auto]{}
                (A0_2);
          
            \draw[>=latex,->]                (A1_0) --
                node[auto]{}
                (A1_1);
          \draw[>=latex,->]                (A1_1) --
                node[auto]{}
                (A1_2);
            \draw[dashed,>=latex,->]                (A1_0) --
                node[auto]{}
                (A0_1);
            
            \path (A0_1) edge [->] node[right,scale=1]
                {$\scriptstyle{p'}$} (A1_1);
             \path (A0_2) edge [->] node[right,scale=1]
                {$\scriptstyle{p}$} (A1_2);
          \path (A0_0) edge [->] node[auto] {$\scriptstyle{}$} (A1_0);
        \end{tikzpicture}
    \end{equation}
If the dashed arrow \(\Spec(R)\to Y'\) in (\ref{homotopy lifting
  property2}) exists, clearly the composition \(\Spec(R)\to Y'\to Y\)
can serve as the dashed arrow in (\ref{homotopy lifting property}).
Thus we can reduce the problem for \(p\) to the problem for \(p'\),
because \cite[Exposé XII, 1.2]{SGA1} and \({p'}^\an\) is a covering space too.

Let \(A'\) be the image of \(A\to R\).  Then we can apply the above reduction
with \(X'=\Spec(A')\).  Thus we may and do assume that \(X\) is
integral, and that \(A\to R\) is injective.

Let \(y \in Y\) be the image of \(\eta \).  It lies above the generic
point of \(X\).  Therefore the residue field \(\kappa(y)\) of \(y\) is a finite
extension of \({\rm Frac}(A)\).  Let \(A''\) be the integral closure of
\(A\) in \(\kappa (y)\).  Then we can apply the above reduction again
with \(X'=\Spec(A'')\).  Observe that the map \(\Spec K \to Y'\)
factors through \(\Spec K \to \Spec \kappa(y)\).  Therefore, after
replacing \(X\) with \(X'\), we may and do assume: \(X\) is normal and
\(\kappa (y)={\rm Frac}(A)\).

% Let $A'$ be the image of the map $A\to R$ corresponding to $(x,\xi)$. By 
%     \cite[Exposé XII, 1.2]{SGA1}, the étale map $p'\coloneq
% p\otimes_AA'$ also has the property that ${p'}^\an$ is a covering space.
% Thus we may assume that $X=\Spec(A')$, \emph{i.e.} we may assume that $A$
% is an integral domain of finite type over $\C$ and that $A\to R$ is injective.

% Suppose that $y\in Y$ is the image of $\eta$, and let $L_{Y,y}$ denote the
% residue field of $y$. Then $L_{Y,y}$ is a finite extension 
% of the function field $L_X$  of $X$. Let $A'\subseteq L_{Y,y}$ be the integral
% closure of $A$ inside $L_{Y,y}$. Then by \cite[Exposé XII, 1.2]{SGA1} again, we
% may assume that $X=\Spec(A')$, \emph{i.e.} we may assume moreover that $A$ is
% normal and $L_{Y,y}=L_X$.

Since $X$ is normal and $p$ is étale, $Y$ is normal. Thus if
$V\subseteq Y$ denote the connected component of $Y$ containing $y$, then $V$
is irreducible. As $p$ is étale, $p(y)\in X$ is the generic point if and only
if $y\in Y$
is a generic point, so $y$ is the generic point of $V$.
By assumption $p^\an$ is a covering space, since $V^\an\subseteq Y^\an$ is
also a connected component, $(p|_V)^\an=p|_{V^\an}$ is also a
covering space of $X^\an$. Thus we may assume that $Y=V$, \emph{i.e.} we may
assume  moreover that $Y$ is an integral normal scheme, and that \(Y\)
and \(X\) has the same function field.

At this moment, we don't know whether \(Y\) is of finite type so we
may not be able to apply  Zariski's Main Theorem (in the form of
\cite[\S 4.4, Corollary 4.6]{Liu02}) to \(p\).  But we can apply it to
$p|_V$ for any quasi-compact open \(V \subseteq Y\), and deduce that
\(p|_V\) is an open embedding for any such \(V\).  We claim that \(p\) itself is an open embedding.
If $p(s)=p(t)$ for $s,t\in Y$, then we can find open affine neighborhoods $V_s,V_t$ of
$s,t$ respectively so that $p|_{V_s}$, $p|_{V_t}$ are open embeddings. As
$V\coloneq V_s\cup V_t$ is quasi-compact, $p|_{V}$ is also an open embedding.
Thus $s=t$, so $p$ is injective. This proves the claim.

As $Y$ is non-empty, so is $Y^\an$. Since $X^\an$ is connected
(cf.~\cite[Exposé XII, Proposition 2.4]{SGA1}), $p^\an$ is surjective. Thus by \cite[Exposé XII, Proposition 3.2 (i)]{SGA1}, $p$ is surjective, or
equivalently, $p$ is an isomorphism. Now we see clearly that \eqref{homotopy lifting property} has a unique
lift, as desired.
\end{proof}

\section{Full Faithfulness}

\begin{thm}\label{full faithfulness} Let $X$ be a scheme locally of finite type over $\C$. Then the
    analytification functor $(-)^\an$ restricting to the category  of geometric covers $\cov(X)$ is fully faithful.

    Moreover, the full subcategory $\Cov(X)\subset \Cov(X^\an)$ is stable
    under taking ``subobjects'' and ``quotients'', \emph{i.e.} if $p\colon Y\to X$
is a geometric cover and if $q'\colon U'\to X^\an\in\Cov(X^\an)$  admits either an injection $a\colon U'\hookrightarrow Y^\an$ or a surjection
$b\colon Y^\an\twoheadrightarrow U'$ over $X^\an$, then
the covering space
$q'$ is algebraizable. 
\end{thm}

\begin{proof} For the first statement we assume, without loss of generality, that $X$ is connected.
    Then both $\Cov(X)$ and $\Cov(X^\an)$ are infinite Galois
    categories. By \cite[Exp. XII Proposition 2.4]{SGA1} $(-)^\an$ sends
    connected objects to connected objects, so the full
    faithfulness follows from
    \cite[Proposition 2.64]{Lara19}.

By the full faithfulness, the second statement is Zariski local on $X$, so we may assume that $X$ affine and connected.

Suppose that there is an injection $a\colon
    U'\hookrightarrow Y^\an$. Then $a$ is étale, so it is an open
    embedding. Thus we can identify $U'$ as an open sub analytic space of
    $Y^\an$ via $a$. Now the task is to show that $U'$ is a Zariski
    open. If $Y=\displaystyle\coprod_{i\in I}Y_i$, where $Y_i$ is a finite
    étale cover over $X$, then $U_i'\coloneq U'\bigcap Y_i^\an$ is a finite covering
    space over $X^\an$, so it is algebraizable by the Riemann Existence Theorem.
    Therefore, $U'=\displaystyle\coprod_{i\in I}U_i'$ is also
    algebraizable. In the general case, we consider the normalization
    map $f\colon \tilde{X}\to X$. Set $\tilde{Y}\coloneq
    \tilde{X}\times_XY$ and $\tilde{U}'\coloneq
    \tilde{X}^\an\times_{X^\an}U'$.  Since $\tilde{Y}$ is a disjoint
    union of finite étale covers of $\tilde{X}$, $\tilde{U}'$ is algebraizable, so it comes
    from a Zariski open $\tilde{U}\subseteq \tilde{Y}$ by \ref{full
    faithfulness}. Then the complement
    \[
        Y^\an\setminus{U'}=f^\an(\tilde{Y}^\an\setminus{\tilde{U}'})=f(\tilde{Y}\setminus{\tilde{U}})^\an
    \]
    is Zariski-closed. Thus $U'$ is Zariski-open.

Suppose that there is a surjection $b\colon Y^\an\twoheadrightarrow U'$. If
$Y=\displaystyle\coprod_{i\in I}Y_i$, where $Y_i$ is finite étale and connected, then
$b(Y_i^\an)\subseteq U'$ is both open and closed. Indeed, $b(Y_i^\an)$
is open because $b$ is étale, and $b(Y_i^\an)$ is closed because it is the
image of the graph $Y_i^\an\subseteq Y_i^\an\times_{X^\an}U'$ of
$b|_{Y_i^\an}$ under the
proper map $Y_i^\an\times_{X^\an}U'\to U'$. Note that
the graph of $b|_{Y_i^\an}$ is closed since $q'\colon U'\to X^\an$ is
separated. In this way, $U'$ becomes a disjoint union of connected
components of the form $b(Y_i^\an)$ which are finite covering spaces. Thus
$U'$ is algebraizable by the Riemann Existence Theorem. In the general
case, one considers the normalization map $f\colon \tilde{X}\to X$. Since the
descent data of $q'\colon U'\to X^\an$ along $f^\an$ are all algebraic,  by the proper
descent of $\Cov(X)$, we find $q\colon U\to X$ in $\Cov(X)$ whose
analytification is $q'$, \emph{i.e.} $q'=q^\an$ is
algebraizable.
\end{proof}

\section{The Comparison Map}

\begin{thm}\label{comparison map} Let $X$ be a connected scheme locally of finite
    type over $\C$, and let $x_0\in X$ be a geometric point. Then there is a
    group
    homomorphism with dense image \[c_X^{\tp\to\pet}\colon
    \pi_1^\tp(X,x_0)\arr\pi_1^\pet(X,x_0)\] making the following diagram
    \[
    \begin{tikzpicture}[xscale=2.9,yscale=-2.2]
        \node (A0_0) at (0, 0) {$\pi_1^\tp(X^\an,x_0)$};
        \node (A0_2) at (2, 0) {$\pi_1^\pet(X,x_0)$};
        \node (A1_1) at (1, 1) {$\pi_1^\et(X,x_0)$};

        %\draw[>=stealth,dashed,->]([yshift= 2pt] A0_0.east) --node[anchor=south]{$w$}
        %    ([yshift= 2pt] A0_1.west);
        %\draw[>=latex,->] ([yshift=-2pt] A0_0.east) --
        %    ([yshift=-2pt] A0_1.west);

        %\draw[>=latex,->] ([yshift= 2pt] A1_0.east) --
        %    ([yshift= 2pt] A1_1.west);
        %\draw[>=latex,->] ([yshift=-2pt] A1_0.east) --
        %    ([yshift=-2pt] A1_1.west);

        \path (A0_0) edge [->] node[below,scale=1]
            {$\scriptstyle{c_X^{\tp\to\et}}\hspace{25pt}$} (A1_1);
        \path (A0_0) edge [>=stealth,dashed,->] node[auto]
            {$\scriptstyle{c_X^{\tp\to\pet}}$} (A0_2);
        \path (A0_2) edge [->] node[auto]
            {$\scriptstyle{c_X^{\pet\to\et}}$} (A1_1);

    \end{tikzpicture}
\]
commutative, where $c_X^{\tp\to\pet}$ and $c_X^{\pet\to\et}$ are the natural
comparison maps. 

Moreover, for each discrete quotient
$u\colon\pi^\pet(X,x_0)\twoheadrightarrow G$,
\emph{i.e.} a surjective homomorphism where $G$ is a discrete group,
the composition $u\circ c_X^{\tp\to\pet}$ is also surjective. In particular, there is no nontrivial continuous
homomorphism from $\pi_1^\pet(X,x_0)$ to a discrete group $G$ whose
restriction to $\pi_1^\tp(X^\an,x_0)$ is trivial.\end{thm}

\begin{proof}
    The existence of $c_X^{\tp\to\pet}$ and the commutativity of the diagram
    follow readily from the infinite Galois theory à la Bhatt-Scholze
    (cf.~\cite[\S 7.2 and \S 7.4]{BS15}).

    The fact that $c_X^{\tp\to\pet}$ has dense image follows from
    \cite[Proposition 2.64]{Lara19} and \ref{full faithfulness}.

    For the last statement we can view $G$ as a
    set equipped with a transitive $\pi_1^\pet(X,x_0)$-action. By \ref{full faithfulness} any
    $\pi_1^\tp(X^\an,x_0)$-equivariant subset of $G$ is
    also a $\pi_1^\pet(X,x_0)$-equivariant subset. Thus $G$ is also a transitive
    $\pi_1^\tp(X^\an,x_0)$-set, \emph{i.e.} $u\circ c_X^{\tp\to\pet}$
    is surjective as desired.
\end{proof}

\printbibliography
\end{document}